\documentclass{amsart}
\usepackage{amsmath,amscd,amstext,amsthm,amsfonts,latexsym}
\usepackage[dvips]{graphicx}

\theoremstyle{plain}\usepackage{amsmath,amssymb,amsthm} 

\newtheorem{theorem}{Theorem}

\newtheorem{lemma}[theorem]{Lemma}
\newtheorem{proposition}[theorem]{Proposition}
\newtheorem{corollary}[theorem]{Corollary}
\newtheorem{definition}[theorem]{Definition}
\newtheorem{remark}[theorem]{Remark}
\newtheorem{example}[theorem]{Example}

\newcommand{\N}{\mathbb{N}}
\newcommand{\R}{\mathbb{R}}

\renewenvironment{proof}[1][.]{%
\bigskip\noindent{\bf Proof#1 }}{%
\hfill$\blacksquare$\bigskip}

\newcommand{\de} {\delta}       

\newcommand{\vep}{\varepsilon}

\newcommand{\cP}{{\mathcal P}}

\begin{document}
\title[Dynamics and Entropy of the Push-Forward]
      {On the Dynamics and Entropy of the Push-Forward Map}

\author[A. Baraviera E. Oliveira  and F. B. Rodrigues]
       {  }


 \email{baravi@mat.ufrgs.br}
 \email{oliveira.elismar@gmail.com}
 \email{patropy@hotmail.com}

\thanks{The first author is partially supported by CNPq, CAPES and FAPERGS}

\maketitle

\centerline{\scshape A. Baraviera, E. Oliveira and F. B. Rodrigues}


\medskip
{\footnotesize
 \centerline{Instituto de Matem\'atica-UFRGS}
   \centerline{ Avenida Bento Gon\c{c}alves 9500 Porto Alegre-RS Brazil}
} 

\bigskip



\begin{abstract} In this work we  study the main dynamical properties of
the push forward map, a transformation in the space of probabilities $\cP(X)$
induced by a map $T:X\rightarrow X$, $X$ a compact metric space.
We also establish a connection between  topological entropies of $T$ and of the push forward map.
\end{abstract}
\maketitle

\section{Introduction}

  During the last years some effort has been made in order to
  endow a probability space (of a given metric space) with
  a Riemannian manifold structure. One of the ingredients
  is the notion of a tangent space, that need to be
  defined in this case, and this motivates, for example,
  the work of Kloeckner. This author fix a certain
  metric space (the circle) and a map on this
  set (a dynamical system); this map induces a transformation
  on the probability space, known as the push forward map,
  and he is able to show some dynamical properties of this map as,
  for example, the entropy and he uses this special case
  in order to give a description of the tangent space of
  the probabilities of the circle.

   Motivated by this work, we start to try to understand
   the relation between a dynamical systems on a compact metric
   space and the dynamical system induced on the probabilities:
   more specifically, to try to know which are the properties
   that are common to both transfomations.

    Some topological properties are inherited by the
  probability dynamics, but with certain losses: for example,
  in order to get transitivity in the probabilistic
  setting it is necessary to assume a very strong
  hypothesis, say, topological mixing for the map
  on the metric space.

   The topological entropy, on the other hand, can
   be bounded below by the entropy of the map $T$ and,
   if its positive, then the topological entropy
   on the probabilistic setting is  indeed infinity.

   The article is organized as follows: after giving
   the main concepts we present some results in a
   simple setting, assuming that the metric space
   is discrete. After this warm up we deal with some
   topological properties of the map $\Phi$ and, at the
   section \ref{entropia topologica}, we address the question of topological entropy.

\section{Motivation: the discrete case}
The goal of the present work is to study the dynamics of the push forward map which arises from a
continuous map $T: X\to X$, i.e., the map $\Phi \colon \cP(X) \to \cP(X)$ given by
$\Phi(\mu) := \mu \circ T^{-1}$.
As a first case it is natural to consider the situation where $X$ is a finite
set or a discrete infinite set. In that case we see that the map $T$ can be represented by a matrix, that we call $[T]$,
and the push forward map $\Phi:\mathcal{P}(X)\to \mathcal{P}(X)$ is then given by the adjoint of the matrix
$[T]$, i.e., $[\Phi]=[T]^{\ast}$.


\subsection{The finite case }
In this section we are going to consider finite spaces. We notice that in these cases X is not connected.
We consider $X=\{x_1,...,x_n\}$, and we identify a function $f: X \to \mathbb{R}$
as a vector in $\mathbb{R}^{n}$ by the linear isomorphism $\mathcal{L}: C^0(X) \to \mathbb{R}^{n}$ given by
$$
  \mathcal{L}(f)= (f(x_1), ..., f(x_n)).
$$
Then
$$
  C^0(X)=\{f:X\rightarrow \R | \; f \mbox{ is continuous}\}\cong \mathbb{R}^n,
$$
and it implies that $C^0(X)^{\prime}\cong (\mathbb{R}^{n})^{*}$ whose basis will be the dual of the canonical one
$$
 \{\delta_{x_{1}}, ..., \delta_{x_{n}}\},
$$
i.e
$$
 \int f d\delta_{x_{i}}= f(x_{i})\text{, for }i=1, ..., n.
$$
by  the identification
$$
  \mathcal{L}^{*}(\nu)=\mathcal{L}^{*}(\sum_{i=1}^n p_i\delta_{x{i}})= (p_1, ..., p_n).
$$

As $\mathcal{M}(X)\cong C^0(X)^{\prime}$,
where $\mathcal{M}(X)$ is the set of measures on $X$, then
\begin{align*}
\mathcal{P}(X)&=\Big\{\sum_{i=1}^np_i\delta_{x{i}}:0\leq p_i\leq0 \mbox{ and }\sum_{i=1}^np_i=1\Big\}
\\&
\cong\Big\{(p_1,...,p_n)\in\mathbb{R}^n:0\leq p_i\leq0 \mbox{ and } \sum_{i=1}^n p_i=1\Big\}.
\end{align*}
So, in that case, the push forward of $T$, i.e., the transformation  $\Phi$, is a map on the simplex
$\Delta_n:=\Big\{(p_1,...,p_n)\in\mathbb{R}^n:0\leq p_i\leq0 \mbox{ and } \sum_{i=1}^np_i=1\Big\}$,
$\Phi:\Delta_n\rightarrow \Delta_n$.

Given $T:X \to X$ a continuous map, we can set a $n\times n$ matrix of zero-one entries $[T]$, that represents $T$ as follows:
$$
 [T]\left(
       \begin{array}{c}
         x_{1} \\
         \vdots \\
         x_{n} \\
       \end{array}
     \right)=\left(
               \begin{array}{c}
                 T(x_{1}) \\
                 \vdots \\
                 T(x_{n}) \\
               \end{array}
             \right)
$$
where
$$
 [T]_{ij}= \left\{
              \begin{array}{ll}
                1, &\text{ if }T(x_i)=x_j\\
                0, & \text{ otherwise.}
              \end{array}
            \right.
$$
We can also identify the integrals in the original space with the inner product:
$$
 \int f d\nu= \langle \mathcal{L}(f), \mathcal{L}^{*}(\nu)\rangle=\sum_{i=1}^n p_i f(x_i).
$$

In order to establish a matrix for the push forward of $T$ we recall the formula of
change of variables (see Lemma~\ref{Change of variables}),
$$
 \int f\circ T d\nu=\int f d \Phi(\nu).
$$

We claim that $\mathcal{L}(f\circ T)= [T] \mathcal{L}(f)$. Indeed, for any $1 \leq i \leq n$ we have the coordinate
$$
 (\mathcal{L}(f\circ T))_{i}= f(x_j)\text{ if }T(x_i)=x_j,
$$
that is
$$
  (\mathcal{L}(f\circ T))_{i}= f(x_j)= f(T(x_i))=([T] \mathcal{L}(f))_{i},
$$
proving the equality.

From this, we have proved the following.
\begin{proposition}\label{Change of variables discrete}
    Let $\Phi$ be the push forward map associated to $T$ and $[\Phi]$ his matrix as above i.e,
    if $\nu=\sum_{i=1}^n p_i\delta_{x{i}}$ then
$$
 \Phi\Big(\sum_{i=1}^n p_i\delta_{x{i}}\Big)=\sum_{i=1}^n q_i\delta_{x{i}}
    \Leftrightarrow [\Phi]\mathcal{L}^{*}(\nu) = \mathcal{L}^{*}(\Phi(\nu)).
$$
Hence, $[\Phi]= [T]^{\ast}$ (the adjoint matrix).
\end{proposition}
\begin{proof}
We just observe that, the change of variables
$$
 \int f\circ T d\nu=\int f d \Phi(\nu),
$$
is equivalent to
$$
  \langle \mathcal{L}(f \circ T), \mathcal{L}^{*}(\nu)\rangle  =
      \langle \mathcal{L}(f ), \mathcal{L}^{*}(\Phi(\nu))\rangle.
$$
We have proved that $\mathcal{L}(f\circ T)= [T] \mathcal{L}(f)$, so
$$
 \langle \mathcal{L}(f \circ T), \mathcal{L}^{*}(\nu)\rangle=\langle
  [T] \mathcal{L}(f), \mathcal{L}^{*}(\nu)\rangle= \langle  \mathcal{L}(f), [T]^{t} \mathcal{L}^{*}(\nu)\rangle,
$$
so
$$
 \langle  \mathcal{L}(f), [T]^{t} \mathcal{L}^{*}(\nu)\rangle =
   \langle \mathcal{L}(f ), \mathcal{L}^{*}(\Phi(\nu))\rangle, \; \forall f (\text{ i.e }\forall \mathbb{R}^{n}),
$$
thus we get $[\Phi]= [T]^{\ast}$.
\end{proof}

\begin{example}
 We consider $X=\{x_0,x_1,...,x_{n-1}\}$ and the map $T:X\rightarrow X$ given by
$$
 T(x_i)=x_{i+1 \mbox{ mod n}}.
$$
Then the matrix of  $T$ is
$$
 [T]=\left[
        \begin{array}{cccccc}
          0 & 0 & 0&\ldots & 0 & 1 \\
          1 & 0 & 0&0 & \ldots & 0 \\
          0 & 1 & 0 & 0 & \ldots&0 \\
          \cdots & \cdots & \cdots& \cdots & \cdots&\cdots \\
          0 & \ldots & 0 & 0 & 1&0 \\
        \end{array}
      \right].
$$
 Given  $\nu=\sum_{i=0}^{n-1}p_i\delta_{x_i}\in\mathcal{P}(X)$,
 if $\Phi:\mathcal{P}(X)\rightarrow \mathcal{P}(X)$ is the push forward of $T$, then
$$
\Phi(\nu)=\sum_{i=0}^{n-1}p_i\delta_{x_i+1 \mbox{ mod n}}.
$$
If we consider $\nu$ as the vector $\nu=(p_0,...,p_{n-1})$, we see that $\Phi(p_0,...,p_{n-1})=(p_{n-1},p_1,p_2,...,p_0)$. Then we conclude that
$$
 [\Phi]=\left[
        \begin{array}{cccccc}
          0 & 1 & 0 & 0 &  \ldots & 0 \\
          0 & 0 & 1 & 0 &  \ldots & 0 \\
     \vdots & 0 & 0 & 1 &  \ldots&0 \\
          0 &\vdots& \ddots& \ddots & 1&\vdots \\
          1 & 0 & \ldots & 0 & 0&0 \\
        \end{array}
      \right]
    = [T]^*
 $$
\end{example}

For the analogous of maps of degree $d$ on $S^1$ we have:
\begin{example}\label{example where T is not a bijection}
Let $X=\{x_0,x_1,x_2,x_3\}$ and $T:X\rightarrow X$, given by
$$
 T(x_i)=x_{2i\mbox{ mod }4}.
$$
Then we have that $T(X)=\{x_0,x_2\}$. Given $\nu=\sum_{i=0}^3p_i\delta_{x_{i}}\in\mathcal{P}(X)$,
if $\Phi:\mathcal{P}(X)\rightarrow \mathcal{P}(X)$ is the push forward of $T$, we can see that
$$
 \Phi(\nu)=(p_0+p_2)\delta_{x_{0}}+(p_1+p_3)\delta_{x_{2}}.
$$
If we consider the measure $\nu$ as the vector $[\nu]=(p_0,p_1,p_2,p_3)^{t}$ then \\
$$
 \Phi(\nu)=[\Phi][\nu],
$$
where

$$
 [\Phi][\nu]=\left[
              \begin{array}{cccc}
                1 & 0 & 1 & 0 \\
                0 & 0 & 0 & 0 \\
                0 & 1 & 0 & 1 \\
                0 & 0 & 0 & 0 \\
              \end{array}
            \right]
\left[
     \begin{array}{c}
       p_0 \\
       p_1 \\
       p_2 \\
       p_3 \\
     \end{array}
   \right].
$$
So, we get $[\Phi]$, which is equal to the adjoint $[T]^{*}$.
\end{example}

\subsection{The infinite case }
In this section we will consider a set $X$  infinite and discrete.
In that case we know that $X=\{x_1,x_2,...\}$ is a  countable set. We endow $X$ with the discrete topology.
We have that the distance on $X$ given by
$$
 d(x_n,x_m)=\left\{
               \begin{array}{cc}
                 1, & \mbox{ if  } n\not=m\\
                 0, & \mbox{ otherwise} \\
               \end{array}
             \right.,
$$
generates the discrete topology on $X$, and with this topology $X$ is not compact.
It is not difficult to see that the set of probability measures on $X$ is given by
$$
 \mathcal{P}(X)=\Big\{\sum_{i=1}^{\infty}p_{i}\delta_{x_{i}}: x_i\in X,\ p_i\geq0,
   \ \sum_{i=1}^{\infty}p_{i}=1\Big\},
$$
and it is also a non compact set.

Let us consider a map $T:X\rightarrow X$ and $\Phi:\mathcal{P}(X)\rightarrow \mathcal{P}(X)$ its push forward.
As in the finite case, we can associate to $T$ a zero-one  matrix, but now it is  an infinite matrix.
Again, if $[T]$ is the matrix associated to the map $T$ we have that the matrix associated to $\Phi$
satisfies the condition $[\Phi]=[T]^{\ast}$.


As $\mathcal{P}(X)$ is convex but not compact, we can not apply the \emph{Schauder Fixed Point Theorem}, but we have the following:

\begin{theorem}\label{Ponto fixo para X discreto infinito}
Let $T:X\rightarrow X$ be a map and $\Phi:\mathcal{P}(X)\rightarrow\mathcal{P}(X)$ its push forward. Then
$T$ has a periodic point if and only if $\Phi$ has a fixed point.
\end{theorem}
\begin{proof}
If there exists $p\in X$ and $n\in\mathbb{N}$ such that $T^n(p)=p$,
then we have that
$$
 \mu = \displaystyle\frac{1}{n}\Big(\delta_{x}+\delta_{T(x)}+...+\delta_{T^{n-1}(x)}\Big)\in\mathcal{P}(X)
$$
 is a fixed point to $\Phi$.

 For the converse, we will divide in two cases, the first one where  $T$ is a bijection.
  It is also possible to think on $T$ as a map
 from
 $\mathbb{N}$ to $\mathbb{N}$, i.e, $T:\mathbb{N}\rightarrow \mathbb{N}$ by means
 of the identification $T(x_i)=x_j  \leftrightarrow T(i)=j $.
 If $\mu=\sum_{i=1}^{\infty}p_{i}\delta_{i}$
 is such that $\Phi(\mu)=\mu$, then
 $$
  \mu = \sum_{i=1}^{\infty}p_{i}\delta_{i}=\sum_{i=1}^{\infty}p_{i}\delta_{T(i)}=\sum_{i=1}^{\infty}p_{T^{-1}(i)}\delta_{i}
 $$
   So we have that $p_{i}=p_{T^{-1}(i)}$, for all $i\in\mathbb{N}$.
 As $\mu$ is a probability measure there exists $p_{j}\not=0$.  Since
 $\Phi(\mu)=\mu$, $\Phi^k(\mu)=\mu$ and it implies that
 $p_{j}=p_{T^{-k}(j)}$ for all $k\in\mathbb{N}$. If $T^{-k}(j)\not=j$ for all $k\in\mathbb{N}$,
 we have that the set $\{T^{-k}(j)=j_k:  k\in\mathbb{N}\}$ is an infinite subset
 of $\mathbb{N}$, and we can write $\mu $ as the following
 $$
 \mu=
 \sum_{k=1}^{\infty}p_{j_k}\delta_{j_k}+\sum_{i\not=j_k \ \forall k }^{\infty}p_{i}\delta_{i}=
 \sum_{k=1}^{\infty}p_{T^{-k}(j)}\delta_{j_k}+\sum_{i\not=j_k \ \forall k }^{\infty}p_{i}\delta_{i}
 =\sum_{k=1}^{\infty}p_{j}\delta_{j_k}+\sum_{i\not=j_k \ \forall k }^{\infty}p_{i}\delta_{i},
 $$
 and it implies that $\mu(\mathbb{N})=\infty$, which is a contradiction.

 For the second case we suppose $T$ is a non bijective map;
 again we can think on $T$ as a map from $\mathbb{N}$ to $\mathbb{N}$.
 Let $\mu=\sum_{i=1}^{\infty}p_{i}\delta_{i}\in\mathcal{P}(X)$ the fixed point of $\Phi$.
 As $\Phi^{k} (\mu)=\mu$ for all $k\in\mathbb{N}$, we have that
 $$
  \mu = \sum_{i=1}^{\infty}p_{i}\delta_{i}=\Phi^{k} (\mu)=\sum_{i=1}^{\infty}p^{k}_{i}\delta_{i},
 $$
 where  $p^{k}_{i}=\sum_{l=1}^{\infty}p_{i^{k}_{l}}=p_{i}$ is given by the set
 $T^{-k}(i)=\{i^{k}_{1},i^{k}_{2},i^{k}_{3},...\}$. We know that there exists $p_j$ such that
 $p_j\not=0$. If $T^{-n}(j)\cap T^{-m}(j)\not=\emptyset$ with $m<n$, then there exists
 $i\in\mathbb{N}$ such that $T^{n}(i)=T^{m}(i)$, and it implies that $T^{n-m}(T^{m}(i))=T^{m}(i)$,
 i.e, $T^{m}(i)$ is a periodic point for $T$. If $T^{-n}(j)\cap T^{-m}(j)=\emptyset$ with $m\not=n$,
  then we can write $\mu$ as the following
$$
\mu=\sum_{j^{1}_{l}\in T^{-1}(j)}^{\infty}p^{1}_{j^{1}_{l}}\delta_{j^{1}_{l}}+
\sum_{j^{2}_{l}\in T^{-2}(j)}^{\infty}p^{2}_{j^{2}_{l}}\delta_{j^{2}_{l}}
+\sum_{j^{3}_{l}\in T^{-3}(j)}^{\infty}p^{3}_{j^{3}_{l}}\delta_{j^{3}_{l}}
+...+\sum_{i\not\in T^{-k}(j), \ \forall k\in\mathbb{N}}^{\infty}p_{i}\delta_{i}.
$$
It implies that $\mu(X)=\infty$, because
$\displaystyle\sum_{j^{k}_{l}\in T^{-k}(j)}^{\infty}p^{k}_{j^{k}_{l}}=p_{j}$ for all
$k\in\mathbb{N}$, and it is a contradiction. Then there exist
$m,n\in\mathbb{N}$ such that $T^{-n}(j)\cap T^{-m}(j)\not=\emptyset$, and by the above, we get that $T$ has a periodic point.
\end{proof}

\begin{example}
Let $X=\{x_1,x_2,...\}$, $x_i\not=x_j$ for $i\not=j$,  and $T:X\rightarrow X$ given by $T(x_i)=x_{i+1}$.
Then, since $T$ has no periodic point, by Theorem
\ref{Ponto fixo para X discreto infinito} we see that $\Phi$ has no fixed point.
\end{example}

\section{The push forward  map $\Phi$ and some metrics on  $\mathcal{P}(X)$ }

Let $X$ a connected compact separable metric space.
If we consider a continuous map $T \colon X \rightarrow X$ it induces a map
$$
 \Phi: \mathcal{P}(X)\rightarrow \mathcal{P}(X),
$$
where $\Phi(\mu)(A)=\mu(T^{-1}(A))$. This map is called {\it the push forward of} $T$.
We are interested in the study of the dynamics of the map $\Phi$.
To do it we observe that there are metrics on $\mathcal{P}(X)$, whose make this set a compact metric space,
since $X$ is also compact.
\begin{proposition}
If we consider $\mathcal{P}(X)$ with the weak topology and T is continuous,
$\Phi$ is continuous. If $T$ is an homeomorphism  then $\Phi$ is an homeomorphism.\\
\end{proposition}
\begin{proof}
See \cite{Gigli}.
\end{proof}

  We are interested in three  particular metrics on $\mathcal{P}(X)$. The first one is the
Prokhorov metric,  defined by
$$
d_{P}(\nu,\mu)=\inf\{\alpha>0:\mu(A)\leq\nu(A_{\alpha})+\alpha
 \mbox{ and } \nu(A)\leq\mu(A_{\alpha})+\alpha ,\ \ \forall A\in \mathcal{B}(X)
\},
$$
where $A_{\alpha}:=\{x\in X: d(x,A)<\alpha\}$. The second one is the the\emph{ weak-$\ast$
distance} (on a locally compact metric space) defined by
$$
 d(\mu,\nu)=\sum_{i=1}^{\infty}\frac{1}{2^{i}}\Big|\int_{X}g_{i}(x)d\mu-\int_{X}g_{i}(x)d\nu\Big|,
$$
where $g_{i}:X\rightarrow [0,1]$ is continuous for all
$i\in\mathbb{N}$ and $\{g_{i}\}_{i\in\mathbb{N}}$ is an enumerable dense set in $C(X,[0,1])$.
The last one is the \textit{Wasserstein metric},  defined by
$$
 W_p(\mu,\nu)=\Big(\inf_{\Pi}\Big\{\int_{X\times X}d^p(x,y)d\Pi\Big\}\Big)^{\frac{1}{p}},
$$
where $\Pi$ is a transport from $\mu$ to $\nu$, say,  a probability on $X \times X$
whose marginals are
$\mu$ and $\nu$.

\begin{lemma}\label{lema1}
(i) All the metrics above generates the weak topology,
\\
(ii) If is $X$ is a compact Polish space, then $\mathcal{P}(X)$ with any of the above metrics is a compact Polish space.
\end{lemma}
\begin{proof}
See \cite{Infinite  dimensional Analysis}.
\end{proof}

\section{Basic topological properties of the map $\Phi$}

We start this section observing that $\Phi$ has a fixed point, since $T$ is continuous.
\begin{proposition}\label{fixed point for the push forward}
If T is a continuous map, then $\Phi$ has a fixed point.
\end{proposition}

\begin{proof}
We notice that $\Phi$ is a continuous map and $\mathcal{P}(X)$ is a compact convex set.
By \textit{Schauder fixed point theorem} we have that $\Phi$ has a fixed point.
\end{proof}
\begin{remark}
 Proposition \ref{fixed point for the push forward}
implies that the  set of probability measures on $X$ which are $T$-invariant, denoted by $\mathcal{M}(T,X)$, is not empty.
\end{remark}

\begin{proposition}
Let $T:X\rightarrow X$ and $S:Y\rightarrow Y$ be topologically conjugated dynamical systems. Then
$\Phi:\mathcal{P}(X)\rightarrow\mathcal{P}(X)$ and $\Psi:\mathcal{P}(Y)\rightarrow\mathcal{P}(Y)$
are topologically
conjugated dynamical systems, where
$\Phi$ is induced by $T$ and $\Psi$ is induced by $S$.
\end{proposition}
\begin{proof}
Let $H: X\rightarrow Y$ be the conjugation between $T$ and $S$. Then we have
$$H\circ T= S\circ H.$$
Consider the map $\Sigma:\mathcal{P}(X)\rightarrow\mathcal{P}(Y)$, given by
$\Sigma(\mu)(A)=\mu(H^{-1}(A))$. Then  $\Sigma$ is a homeomorphism. Take $\nu\in\mathcal{P}(Y)$
and see that
\begin{align*}
\Sigma\circ\Phi\circ\Sigma^{-1}(\mu)&=\Sigma\circ\Phi(\mu\circ H)=\Sigma(\mu\circ H \circ T)
\\&=\mu\circ H \circ T\circ H^{-1}= \mu\circ S\circ H\circ H^{-1}
\\&=\mu\circ S=\Psi(\mu)
\end{align*}
Hence
$$\Sigma\circ\Phi\circ\Sigma^{-1}=\Psi,$$
which implies the result.
\end{proof}

Now we define a measurable partition of the set $X$ that we call grid.

\begin{lemma}\label{delta grading}
Given $X$ a compact metric space  and $\delta>0$, there exists a measurable covering of $X$, $\{P_{j}\}_{j=1}^{N}$,
such that each
$P_{j}$ has non-empty interior, $P_{i}\cap P_{j}=\emptyset$ for any $i \neq j$
and $d(x,y)<\delta$ for all $x,y\in P_{j}$, for all j.
Moreover, there exist $\varepsilon>0$ and points $p_i\in P_{i}$ such that $B_{\varepsilon}(p_i)\subset P_{i}$.
\end{lemma}
\begin{proof}
Given $\delta>0$, there exist $x_1,...,x_k \in X$ such
that $\displaystyle X=\cup_{j=1}^k B_{\frac{\delta}{2}}(x_j)$. So we define
\begin{center}
$\displaystyle P_1= B_{\frac{\delta}{2}}(x_1)$,

$\displaystyle P_2= (B_{\frac{\delta}{2}}(x_2))-( B_{\frac{\delta}{2}}(x_1))$

$\vdots$

$\displaystyle P_k= (B_{\frac{\delta}{2}}(x_k))-(\cup_{j=1}^{k-1} B_{\frac{\delta}{2}}(x_j)).$
\end{center}
Then we get $X=\cup_{j=1}^k P_j$, and $P_i\cap P_j=\emptyset$, if $i\not= j$. As $ P_j\subset
B_{\frac{\delta}{2}}(x_j) $ , $d(x,y)<\delta$ for all $x,y\in P_j$.\\
As the covering  $\displaystyle X=\cup_{j=1}^k B_{\frac{\delta}{2}}(x_j)$ is finite
and by construction of each $P_i$, we can take a suitable $\varepsilon>0$ and choose points
$p_i\in P_i$ such that $B_{\varepsilon}(p_i)\subset P_{i}$ for $i\in \{1,...,k\}$.
\end{proof}

 With this grid in mind we can approximate any measure as follows:
\begin{lemma}
Given $\mu \in \cP(X)$  and $\varepsilon > 0$, there exists
$$
   \nu = \sum_{i=1}^N  a_i \delta_{p_i}
$$
such that $d(\mu, \nu) < \varepsilon$.
\end{lemma}
\begin{proof}
 Given $\varepsilon>0$,
 there exists $n_0\in\mathbb{N}$,
 such that
 $$
 \sum_{i=n_0+1}^{\infty}\frac{1}{2^{i}}<\displaystyle\frac{\varepsilon}{2}.
 $$
 Using the continuity of $g_i$,
 there exists $\delta=\delta\Big(n_0,\displaystyle\frac{\varepsilon}{2}\Big)$,
  such that
$$
 d(x,y)<\delta\Rightarrow |g_i(x)-g_i(y)|
 <\displaystyle\frac{\varepsilon}{2}, \ \forall i\in\{1,...,n_0\}.
$$
  Given $\de > 0$, let us consider a grading
 $P=\{P_1, \ldots, P_N\}$ such that
 $\mbox{diam}(P_i)<\delta$ for all $ P_i$.
 Take points $p_i \in P_i$ and consider
 the probability
 $$
   \nu= \sum_{i=1}^N  \mu(P_i) \de_{p_i}.
 $$
 Then we notice that
 \begin{align*}
 d(\nu,\mu)&=\sum_{j=1}^{\infty}\frac{1}{2^j}\Big|\int_X g_j(x)d\nu-\int_X g_j(x)d\mu\Big|
 \\&=\sum_{j=1}^{\infty}\frac{1}{2^j}\Big|\sum_{i=1}^{k}\int_{P_i} g_j(p_i) - g_j(x)d\mu\Big|
 \\&\leq\sum_{j=1}^{\infty}\frac{1}{2^j}\sum_{i=1}^{k}\int_{P_i}\Big|g_j(p_i)-g_j(x)\Big|d\mu
  <\varepsilon
 \end{align*}

\end{proof}

For the next we assume that the homeomorphism $T$ is such that its periodic points
are dense in $X$, i.e.: given $\delta >0$, there exists
a $K-$periodic point $p\in X$ such that its orbit $\{p, T(p),..., T^{K-1}(p)\}$ is  $\delta-$dense.
We can also define periodic measures,
say, measures that are periodic points of the dynamics $\Phi$.

\begin{proposition}
If  $T:X\rightarrow X$ is a homeomorphism with dense periodic points, then the periodic points for $\Phi$ are dense in $\mathcal{P}(X)$.
\end{proposition}
\begin{proof}
 Given any measure $\mu\in \mathcal{P}(X)$, we need to show how it can be approximated by a periodic measure.
Take $\varepsilon>0$, then there exists $n_0$ such that
$$\sum_{i=n_0+1}^{\infty}\frac{1}{2^{i}}<\displaystyle\frac{\varepsilon}{2}.$$
Using the continuity of $g_i$,
there exists $\delta=\delta\Big(n_0,\displaystyle\frac{\varepsilon}{2}\Big)$, such that
$$d(x,y)<\delta\Rightarrow |g_i(x),g_i(y)|<\displaystyle\frac{\varepsilon}{2}, \ \forall i\in\{1,...,n_0\}.$$
 We consider a $\delta-$grid on X,  $P=\{P_1,...,P_K\}$,
 and take a periodic orbit in $X$, $\{p, T(p),..., T^{N-1}(p)\}$,
 which is $\displaystyle\frac{\delta}{2}-$dense.
Clearly, there exists at least one point of the orbit in each element $P_i$ (and so
$K \leq N$). Let us relabel the orbit as follows: call $q_1$ a point lying in $P_1$ (any one of
the finite points in this set can be chosen); $q_i$ some point lying in $P_i$ and so on,
until $q_N\in  P_N$. So we define the measure
$$\mu^{\prime}=\sum_{i=1}^N \mu(P_i)\delta_{q_i}.$$
Then we have that
\begin{align*}
d(\mu,\mu^{\prime})&=\sum_{i=1}^{\infty}\frac{1}{2^i}\Big|\int_Xg_i(x)d\mu-\int_Xg_i(x)d\mu^{\prime}\Big|
\\&=\sum_{i=1}^{\infty}\frac{1}{2^i}\Big|\sum_{j=1}^{K}\int_{P_j}(g_i(x)-g_{i}(q_j))d\mu\Big|
\\&\leq\sum_{i=1}^{\infty}\frac{1}{2^i}\sum_{j=1}^{K}\int_{P_j}|g_i(x)-g_{i}(q_j) |d\mu
\\&=\sum_{i=1}^{n_0}\frac{1}{2^i}\sum_{j=1}^{K}\int_{P_j}|g_i(x)-g_{i}(q_j) |d\mu
+\sum_{i=n_0+1}^{\infty}\frac{1}{2^i}\sum_{j=1}^{K}\int_{P_j}|g_i(x)-g_{i}(q_j) |d\mu
\\&\leq\sum_{i=1}^{n_0}\frac{1}{2^i}\sum_{j=1}^{K}\frac{\varepsilon}{2}\mu(P_j)+
\sum_{i=n_0+1}^{\infty}\frac{1}{2^i}\sum_{j=1}^{K}2\mu(P_j)<\varepsilon,
\end{align*}
where the last inequality comes from the fact $\mu(X)=\sum_{j=1}^{K}\mu(P_j)=1$
\end{proof}

\begin{definition}
Let $T:X\rightarrow X$ a homeomorphism of a compact metric space.
We say that $T$ is equicontinuous if the sequence of iterates of $T$, $\{T^n\}_{n\in \N}$, is an
equicontinuous sequence of homeomorphisms.
\end{definition}

\begin{proposition}
If T is equicontinuous, then $\Phi$ is equicontinuous.
\end{proposition}
\begin{proof}
 Let us suppose $T$ equicontinuous and consider the sequence $\{\Phi^n\}_{n\in \N}$. Take $\varepsilon >0$,
then there exists $\delta>0$ such that
$$
 d(x,y)<\delta\Rightarrow d(T^n(x),T^n(y))<\varepsilon.
$$
By considering the
 Prokhorov metric we have
 $$
  d_P(\mu,\nu)<\delta\Rightarrow d_P(\Phi^n(\mu),\Phi^n(\nu))<\varepsilon, \forall n\in\N.
 $$
 To see that we suppose $d_P(\mu,\nu)<\delta$ and observe that
$$
 (T^{-n}(A))_{\delta}\subset T^{-n}(A_{\varepsilon}),
$$
where $A_{\gamma}=\{x\in X: d(x,A)<\gamma\}$, for some $A\subset X$.
In fact, if $x\in (T^{-n}(A))_{\delta}$, then there exists $z\in T^{-n}(A)$ such that
$d(x,z)<\delta$, but it implies $d(T^n(x),T^n(z))<\varepsilon$.
As $z\in T^{-n}(A)$, $T^n(z)\in A$,  so $T^n(x)\in A_{\varepsilon}$ and it implies
$x\in T^{-n}(A_{\varepsilon})$. Then we have that
\begin{align*}
\Phi^n(\mu)(A)=\mu(T^{-n}(A))\leq \nu((T^{-n}(A))_{\delta})+\delta\leq\nu( T^{-n}(A_{\varepsilon}))+\varepsilon
=\Phi^n(\nu)(A_{\varepsilon})+\varepsilon
\\
\Phi^n(\nu)(A)=\nu(T^{-n}(A))\leq \mu((T^{-n}(A))_{\delta})+\delta\leq\mu( T^{-n}(A_{\varepsilon}))+\varepsilon
=\Phi^n(\mu)(A_{\varepsilon})+\varepsilon,
\end{align*}
 and it implies  $d_P(\Phi^n(\mu),\Phi^n(\nu))<\varepsilon$.
\end{proof}

We also can prove that $T$ Lipschitz implies $\Phi$ Lipschitz. In order to prove that result we need the following:

\begin{lemma}\label{Change of variables}
(Change of variables)
 Let $f:X\rightarrow \R$ be a measurable function and $T:X\rightarrow X$ continuous. Then
 $$\int_X f d(\Phi(\mu))=\int_X (f\circ T)(x)d\mu.$$
\end{lemma}
\begin{proof}
See \cite{Infinite  dimensional Analysis}.
\end{proof}

\begin{proposition}\label{lipschitz}
If $T:X\rightarrow X$ is C-Lipschitz, then $\Phi:\mathcal{P}(X)\rightarrow\mathcal{P}(X)$ is C-Lipschitz
with respect to the Wasserstein metric. If we consider the Prokhorov metric or the weak-$\ast$ metric $\Phi$ is Lipschitz, but C can change.
\end{proposition}
\begin{proof}
 Let us consider the map $(T,T): X\times X\rightarrow X\times X$ defined by
$(T,T)(x,y)=(T(x),T(y))$. We have that $(T,T)$ is continuous, so $(T,T)$ induces a continuos
map on $\mathcal{P}(X\times X)$, let us say $\Psi$.
Hence if $\Pi$ is a measure on $X\times X$ we have, by the Lemma \ref{Change of variables}
 $$\int_{X\times X}d^p(x,y)d(\Psi(\Pi))=\int_{X\times X}d^p(T(x),T(y))d\Pi.$$
  We observe that if $\mu,\nu\in\mathcal{P}(X)$ and $\Pi$ is a transport from $\mu$
to $\nu$ then $\Psi(\Pi)$ is a transport from $\Phi(\mu)$ to $\Phi(\nu)$. Then, if $T$ is
a $C$-Lipschitz function we have

\begin{align*}
W_{p}^p(\Phi(\mu),\Phi(\nu))&=
\inf_{\Pi^{\prime} }\Big\{\int_{X\times X}d^p(x,y)d\Pi^{\prime}:\Pi^{\prime}
 \ \mbox{is a transport from } \ \Phi(\mu) \ \mbox{to}\ \Phi(\nu)\Big\}
\\&\leq
\inf_{\Pi}\Big\{\int_{X\times X}d^p(x,y)d(\Psi(\Pi)): \Pi  \mbox{ is a transport from }  \mu \ \mbox{to}\ \nu\Big\}
\\&=\inf_{\Pi}\Big\{\int_{X\times X}d^p(T(x),T(y))d(\Pi):\Pi  \mbox{ is a transport from }  \mu \ \mbox{to}\ \nu\Big\}
\\&\leq C\inf_{\Pi}\Big\{\int_{X\times X}d^p(x,y)d(\Pi):\Pi  \mbox{ is a transport from }  \mu \ \mbox{to}\ \nu\Big\}
\\&=C W_{p}^p(\mu,\nu).
\end{align*}
 Since the Prokhorov metric and the weak-$\ast$ are equivalents to the Wasserstein metric, we get the result.
\end{proof}


Another natural question is whether transitivity of $T$ implies transitivity
of $\Phi$. The example below shows that the answer is {\em negative}.

\begin{remark}\label{ not transitive}
T transitive does not imply $\Phi$ transitive.
\end{remark}
\begin{proof}
If $T:\mathbb{S}^1\rightarrow\mathbb{S}^1$ is the irrational rotation
on the circle given by $T(x)=x+\alpha$, $\alpha$  an irrational number,
we have that $T$ is transitive. As $T$ is a translation, we have that
$\Phi$ is 1-Lipschitz, if we consider on $\mathcal{P}(X)$ the Prokhorov distance.
  If we assume $\Phi$ transitive we have that there exists $\mu\in\mathcal{P}(X)$
  such that the forward orbit $\{\Phi^n(\mu):n\in\mathbb{N}\}$ is dense in $\mathcal{P}(X)$.
  Take $\varepsilon>0$ such that $0\not\in\displaystyle A=\Big(\varepsilon,1-\varepsilon\Big)$ and
  $1-2\varepsilon>\varepsilon$ (what corresponds to a choice of $\vep \in (0, 1/3)$).
  Consider the Lebesgue measure $\lambda\in\mathcal{P}(X)$,
  there exists $n\in\mathbb{N}$, such that $\displaystyle d_P(\Phi^{n}(\mu),\lambda)<\frac{\varepsilon}{4}.$
    Take  $p=0\in\mathbb{S}^1$. By the density of the sequence $\{\Phi^{k}(\mu)\}_{k\in\mathbb{N}}$,
  there exists $l\in \mathbb{N}$, such that
  $\displaystyle d_P(\Phi^{n+l}(\mu),\delta_0)<\frac{\varepsilon}{4}.$
  As $\Phi$ is 1-Lipschitz and $\lambda$ is $\Phi$-invariant,
 we have that
 $$\displaystyle d_P(\Phi^{n+l}(\mu), \lambda)=d_P( \Phi^{n+l}(\mu), \Phi^{l}(\lambda))
 \leq d_P(\Phi^{n}(\mu),\lambda) < \frac{\varepsilon}{4}.$$
 By triangular inequality we get the following
 $$d_P(\lambda,\delta_0)\leq \displaystyle d_P(\Phi^{n+l}(\mu), \lambda)+ \displaystyle d_P(\Phi^{n+l}(\mu), \delta_0)\leq \frac{\varepsilon}{2}.$$
It implies that
$$\displaystyle\lambda (A)\leq \delta_0(A_{\frac{\varepsilon}{2}})+\frac{\varepsilon}{2},
  \mbox{ and }
  \delta_0(A)\leq \lambda(A_{\frac{\varepsilon}{2}})
  +\frac{\varepsilon}{2}, \ \forall A\in\mathcal{B}(\mathbb{S}^1).$$
  In particular, if $A=\Big(\varepsilon,1-\varepsilon\Big)$, $\displaystyle0\not\in A_{\frac{\varepsilon}{2}}$. Then
  $$1-2\varepsilon=\displaystyle\lambda (A)\leq \delta_0(A_{\frac{\varepsilon}{2}})+\frac{\varepsilon}{2}=\frac{\varepsilon}{2},$$
  which is a contradiction.
 \end{proof}

  We assume now an stronger condition, say,
  that $T$ is topologically mixing, i.e., given $U,V$ open sets in $X$, there exists
$N\in\mathbb{N}$ such that $T^{n}(U)\cap V\neq \emptyset$ for all  $n>N$. We notice that
  $T^{-1}$ is also topologically mixing, since $T$ is bijective.

\begin{proposition}
If $T:X\rightarrow X$ is topologically mixing then $\Phi$ is topologically mixing.
\end{proposition}
\begin{proof}
We notice that given $k\in\N$ we have that the map
$$
 T^k:=(T,...,T):X^k\rightarrow X^k
$$
is topologically mixing if and only if $T$ is topologically mixing.
The proof is left to the reader.

If we take $\mu,\nu\in\mathcal{P}(X)$
and $\varepsilon>0$ and consider the open balls $B(\mu,\varepsilon)$ and
$B(\nu,\varepsilon)$ in $\mathcal{P}(X)$, then there exist
$\mu^{\prime}=\sum_{i=1}^k a_i\delta_{x_{i}}\in B(\mu,\varepsilon)$ and
$\nu^{\prime}=\sum_{i=1}^k b_i\delta_{y_{i}}\in B(\nu,\varepsilon)$.
Taking the points $(x_1,...,x_k),(y_1,...,y_k)\in X^k$ and $\delta >0$ such
that
 $$
   d((u_1,...,u_k),(v_1,...,v_k))<\delta \Rightarrow \sum_{j=1}^{\infty}\frac{1}{2^j}
   \sum_{i=1}^k|f_j(u_i)-f_j(v_i)| \leq\varepsilon_0,
 $$
where $\varepsilon_0$ is such that
$d(\mu,\mu^{\prime})+\varepsilon_0\leq \varepsilon$ and $d(\nu,\nu^{\prime})+\varepsilon_0\leq \varepsilon$.\\
Now we consider the open balls $B((x_1,...,x_k),\delta)$ and $B((y_1,...,y_k),\delta)$
in $X^k$. As $T^k$ is topologically mixing there exists $N\in\N$ such that
$$
  n>N\Rightarrow (T^k)^n(B((x_1,...,x_k),\delta))\cap
  B((y_1,...,y_k),\delta)\not=\emptyset.
$$
Then there exists $(z_1,...,z_k)\in B((x_1,...,x_k),\delta)$, such that $(T^k)^n(z_1,...,z_k)$ is in  $B((y_1,...,y_k),\delta)$. Finally we consider the measure
$\bar{\mu}=\sum_{i=1}^ka_i\delta_{z_{i}}$. As
$$
 d((x_1,...,x_k),(z_1,...,z_k))< \delta \Rightarrow
 \sum_{j=1}^{\infty}\frac{1}{2^j} \sum_{i=1}^k|f_j(x_i)-f_j(z_i)|
 \leq \varepsilon_0,
$$
and
$$
 d((T^n(z_1),...,T^n(z_k)),(y_1,...,y_k)) <
  \delta\Rightarrow \sum_{j=1}^{\infty}\frac{1}{2^j}
   \sum_{i=1}^k|f_j(T(z_i))-f_j(y_i)| \leq  \varepsilon_0,
$$
we get
$$
 d(\bar{\mu},\mu)\leq \varepsilon,  \mbox{ and }
  d(\nu,\Phi^n(\bar{\mu}))\leq \varepsilon.
$$
It implies that $\Phi^n(B(\mu,\varepsilon))\cap B(\nu,\varepsilon)\not=\emptyset.$
\end{proof}

\begin{remark}
  Is is well known  that any topologically mixing continuous transformation
  on a compact set is transitive; then we conclude that $T$ topologically mixing
  implies that $\Phi$ is transitive.
\end{remark}



\section{Limit sets}
In this section we consider some limit sets for the map $T$ and the consequences
on the induced push-forward map.

\subsection{Non-wandering set}

\begin{definition}
Given $p\in X$, $p$ is called non-wandering if for all $U$ neighborhood of $p$ and
$N\in\mathbb{N}$ , there exists $n\in\mathbb{N}$
such that $n>N$ and $T^{n}(U)\cap U\neq\varnothing$.
\end{definition}

\begin{proposition}
If $p\in X$ is non-wandering, then $\delta_{p}$ is non-wandering.
\end{proposition}
\begin{proof}
 Let $p$ be non-wandering. Then given $\varepsilon>0$, there exists $n\in\mathbb{N}$
 such that $T^{n}(B_{\varepsilon}(p))\cap B_{\varepsilon}(p)$, i.e., there exists
 $q\in T^{n}(B_{\varepsilon}(p))\cap B_{\varepsilon}(p)$. Now we take $\delta_{q}$ and notice that
$$
 d_{P}(\delta_p,\delta_q)\leq d(p,q)\Rightarrow \delta_{q}\in B_{\varepsilon}(\delta_{p}),
$$
and as $q\in T^{n}(B_{\varepsilon}(p))$, there exists $x\in B_{\varepsilon}(p)$, such that $q=T^{n}(x)$. Then
$$
 \delta_{q}=\delta_{T^{n}(x)}=\Phi^{n}(\delta_{x})\in \Phi^{n}(B_{\varepsilon}(\delta_{p})).
$$
Finally we conclude that
$\delta_{q}\in B_{\varepsilon}(\delta_{p})\cap\Phi^{n}(B_{\varepsilon}(\delta_{p}))\neq\varnothing.$
\end{proof}

\subsection {$\omega$-limit}
\begin{definition}
Let $T:X\rightarrow X$ a continuous map. Let $x\in X$. A point $y\in X$ is an
$\omega-$limit point if there exists a sequence of natural numbers $n_k\rightarrow\infty$
(as $k\rightarrow\infty$) such that $T^{n_k}(x)\rightarrow y$. The $\omega-$limit set is the set
 $\omega(x)$ of all  $\omega-$limit points.
\end{definition}
\begin{proposition}\label{omega limite}
If $q$ $\in $ $\omega(p)$, then $\delta_{q}\in \omega(\delta_{p})$.
\end{proposition}
\begin{proof}
We need to show that there exists a sequence $\{\Phi^{n_{k}}(\delta_{p})\}_{n_{k}\in\mathbb{N}}$, such that,
 $n_{k}\rightarrow\infty$ and $\Phi^{n_{k}}(\delta_{p})\rightarrow\delta_{q}$. Since $q$ $\in $ $\omega(p)$,
there exists a sequence $\{T^{n_{k}}(p)\}_{n_{k}\in\mathbb{N}}$, such that, $T^{n_{k}}(p)\rightarrow q$.
 Now given $g\in C(X)$ we have that

$$
 \Big|\int_{X}g(x)d(\Phi^{n_{k}}(\delta_{p}))-\int_{X}g(x)d(\delta_{q})\Big|=|g(T^{n_{k}}(p))-g(q)|.
$$
As $g$ is continuous and $T^{n_{k}}(p)\rightarrow  q$, $g(T^{n_{k}}(p))\rightarrow g(q)$. Then we get
$$
 \int_{X}g(x)d(\Phi^{n_{k}}(\delta_{p}))\rightarrow\int_{X}g(x)d(\delta_{q}), \ \ \ \forall g \in C(X).
$$
Hence
$$
 d(\Phi^{n_{k}}(\delta_{p}),\delta_{q})\rightarrow 0.
$$
\end{proof}

\begin{definition}
A point $p \in X$ is called recurrent if $x\in\omega(x)$. The set $\mathcal{R}(T)$ of recurrent points is $T$-invariant.
\end{definition}

Hence, by Proposition \ref{omega limite}, given $x\in\mathcal{R}(T)$, we have that $\delta_{x}\in\omega(\delta_{x})$. Then
$$
 x\in\mathcal{R}(T)\Rightarrow\delta_{x}\in\mathcal{R}(\Phi).
$$

\section{Attractors}
Here we are interested in know what happens with the dynamics $\Phi$ when the dynamics $T$
has an attractor. We divide our study in two cases: the first one consists in a map $T$ that has a point  $p$ as
an attractor and the second one consists in a map that has a uniform attractor.
\subsection{Point attractor}
\begin{lemma}\label{entropia}
Let $T:X\rightarrow X$ be a continuous map such that $T:X\rightarrow T(X)$ is a homeomorphism.
If  $\displaystyle\lim_{n\rightarrow\infty} T^{n}(x)=p$,
 for all $x\in X$, then the sequence of maps $\{T^{n}\}_{n\in\mathbb{N}}$ converges uniformly
 to the constant map $F:X\rightarrow X$, $F(x)=p$ for all $x\in X$.
\end{lemma}
\vspace{-.4 cm}
\begin{proof}
Consider the following sequence of continuous maps $G_{n}=T^{n}:X\rightarrow X $ and the map
$F:X\rightarrow X$ given by $F(x)=p$
for all $x\in X$.  We observe that $G_{n}(x)\rightarrow P$ for all $x\in X$, i.e,
$G_{n}$ converges to $F$ pointwise. As $X$ is compact we have that $G_{n}\rightarrow F$, uniformly.
\end{proof}

\begin{proposition}\label{ponto fixo para medida}
Let $T:X\rightarrow X$ be a continuous map such that $T:X\rightarrow T(X)$ is a homeomorphism. If
 $\displaystyle\lim_{n\rightarrow\infty}T^{n}(x)=p$,
$\forall x \in X$, then $\displaystyle\lim_{n\rightarrow\infty}\Phi^{n}(\mu)=\delta_{p}$, $\forall \mu \in \mathcal{P}(X)$.
\end{proposition}
\vspace{-.4 cm}
\begin{proof}
 Take $\varepsilon>0$. We need to show that there exists $n_0\in \N$ such that
$$
 n>n_0\Rightarrow d(\Phi^n(\mu),\delta_{T^n(p)})<\varepsilon.
$$
  By  Lemma \ref{entropia} we can see that given $\delta>0$, there exists $n_{0}\in\mathbb{N}$,
such that $d(T^{n}(x),p)<\delta$, for all $x\in X$ and $n>n_{0}$.
 Now we take $g\in C(X)$ and see that
\begin{align*}
\Big|\int_{X}g(x)d(\Phi^{n}(\mu))-\int_{X}g(x)d\delta_{p}\Big|&=\Big|\int_{X}(g(T^n(x))-g(p))d\mu\Big|
\\&\leq\int_{X}|(g(T^n(x))-g(p))|d\mu
\\&
\leq \sup_{x\in X} |(g(T^n(x))-g(p))|.
\end{align*}
 Since $g\in C(X)$
we get
$\displaystyle\Big|\int_{X}g(x)d(\Phi^{n}(\mu))-\int_{X}g(x)d\delta_{p}\Big|\rightarrow 0$, for all $g\in C(X)$.
Hence $d(\Phi^{n}(\mu),\delta_{p})\rightarrow0$.
\end{proof}

\subsection{Uniform attractor}
 In this section we define the concept of uniform attractor and see what happens with the dynamics $\Phi$
when $T$ has a uniform attractor. To do that we  suppose that $X$ is separable.

\begin{definition}
Let $\Lambda\subseteq X$ be a compact set such that $T(\Lambda)\subseteq\Lambda$.
We say that $\Lambda$ is a uniform attractor for $T$, if for all $\varepsilon>0$,
there exists $n_{0}\in\mathbb{N}$ such that
$$
  n>n_{0}\Rightarrow d(T^{n}(x),\Lambda)<\varepsilon, \ \forall x\in X.
$$
\end{definition}

\begin{lemma}
Let $T:X\rightarrow X$ be a homeomorphism from $X$ to $T(X)$ and $A=\{a_{j}\}_{j\in\mathbb{N}}$ dense in X.
Then $T^{n}(A)$ is dense in $T^{n}X$, for all $n\in\mathbb{N}$.
\end{lemma}
\begin{proof}
Take $x\in X$ and $\varepsilon>0$, we have to show that for all $n\in\mathbb{N}$, there exists $a_{i}\in A$
such that $d(T^{n}(x), T^{n}(a_{i}))<\varepsilon.$ Since $T^{n}$ is a continuous map, there exists
 $\delta>0$ such that
$$
 d(y,a_{i})<\delta\Rightarrow d(T^{n}(y),T^{n}(a_{i}))<\varepsilon.
$$
Using the density of $A$ in $X$ we get the result.
\end{proof}

\begin{lemma} \label{lema}
\begin{flushleft}
 (i)   If $\mu=\sum^{l}_{i=1}\alpha_{i}\delta_{a_{i}}$, and $\nu=\sum^{k}_{i=1}\beta_{i}\delta_{b_{i}}$, then
$$
 d_{P}(\nu,\mu)\leq \max\{d(a_{i},b_{j})\}.
$$
(ii) If $\mu=\sum^{l}_{i=1}\alpha_{i}\delta_{a_{i}}$ and $\nu=\sum^{l}_{i=1}\alpha_{i}\delta_{b_{i}}$, then
$$
 d_{P}(\nu,\mu)\leq \min\{d(a_{i},b_{i})\},
$$
\end{flushleft}
where $d_P$ is the Prokhorov distance.
\end{lemma}
\begin{proof}
 (i) We take $\gamma>\max\{d(a_{i},b_{j})\}$ and $A\in \mathcal{B}(X)$. Then
$$
 \exists \  a_{i}\in A\Rightarrow  b_{j}\in A_{\gamma}, \  \forall \ j,
 \ \mbox{and} \ \exists
b_{i}\in A\Rightarrow a_{j}\in A_{\gamma}\ \forall j.
$$
Hence we have that
$$
 \mu(A)\leq \nu(A_{\gamma})+\gamma,  \ \ \ \nu(A)\leq \mu(A_{\gamma})+\gamma,
$$
for all $A\in \mathcal{B}(X)$. Then, by the definition of $d_{P}$, we conclude
$$
 d_{P}(\nu,\mu)\leq \max\{ d(a_{i},b_{j})\}.
$$
\begin{flushleft}
(ii) We take $\gamma>\min\{d(a_{i},b_{i})\}$. We notice that
$$
  \exists   a_{i}\in A\Rightarrow  b_{i}\in A_{\gamma}, \ \mbox{and} \ \exists
b_{i}\in A\Rightarrow a_{i}\in A_{\gamma}.
$$
for all $A\in \mathcal{B}(X)$. Then, by the definition of $d_{P}$, we conclude
$$
 d_{P}(\nu,\mu)\leq \min\{ d(a_{i},b_{i})\}.
$$
\end{flushleft}
\end{proof}

\begin{lemma}\label{sepa}
If $X$ is a compact separable metric space then $\mathcal{P}(X)$ is a compact separable metric space.
\end{lemma}
\begin{proof} Let $A$ the enumerable dense set in $X$. Consider
$$
 \mathcal{A}=\Big\{\sum_{i=1}^{k}\alpha_i\delta_{x_i}:\alpha_1,...,\alpha_{i_l}
 \in[0,1] \cap \mathbb{Q}, \ x_i\in A \mbox { and } k\in\mathbb{N} \Big\}.
$$
It is not difficult to see that $\mathcal{A}$ is an enumerable dense set in $\mathcal{P}(X)$.
\end{proof}

\begin{theorem}\label{teorema}
Let $\Lambda\subseteq X$ be an uniform attractor for $T$.  If
 $$
  \mathcal{D}:=\Big\{\sum_{i=1}^{k}\alpha_{i}\delta_{q_{i}}:\sum_{i=1}^{k}\alpha_{i}=1,\
\ q_{i}\in \Lambda, \ \ \alpha_{i}\in [0,1]\cap\mathbb{Q},\ k\in \mathbb{N}\Big\},
$$
then $\overline{\mathcal{D}}$
 is an uniform attractor for $\Phi$.
\end{theorem}
\begin{proof}
 By Lemma \ref{sepa}, we have that
 $$
   \mathcal{A}=\Big\{\sum_{i=1}^{k}\alpha_{i}\delta_{a_{i}}:\sum_{i=1}^{k}\alpha_{i}=1,
    \ \ a_{i}\in A, \ \ \alpha_{i}\in [0,1]\cap\mathbb{Q}\mbox { and } k\in\mathbb{N}\Big\}
 $$
is dense in $P(X)$.
 Using the Lemma \ref{lema}
$$
 \lim_{n\rightarrow\infty}d(\Phi^{n}(\nu),\mathcal{D})=0,
$$
uniformly, for all $\nu\in\mathcal{A}$ .
In fact, if we take $\varepsilon>0$ , there exists $n_{0}\in\mathbb{N}$ such that
$$
 n>n_{0}\Rightarrow d(T^{n}(a_{i}),\Lambda)<\varepsilon, \ \ \forall \ a_{i}\in A.
$$
Given $a_{i}\in A$, there exists $q_{i}\in \Lambda$, such that $d(T^{n}(a_{i}),q_{i})<\varepsilon$.
Hence if $\nu=\sum_{i=1}^{k}\alpha_{i}\delta_{a_{i}}$ and we consider
$\nu^{\prime}=\sum_{i=1}^{k}\alpha_{i}\delta_{q_{i}}$,
where $d(T^{n}(a_{i}),q_{i})<\varepsilon$, we see that, by Lemma \ref{lema},
$$
 d_{P}(\Phi^{n}(\nu),\nu^{\prime})<\min\{d(T^{n}(a_{i}),q_{i})\}<\max\{d(T^{n}(a_{i}),q_{i})\}<\varepsilon.
$$
Now we take $\mu\in P(X)$ and $\varepsilon>0$. We know that there exists $n_{0}\in\mathbb{N}$ such that
$$
 n>n_{0}\Rightarrow d_{P}(\Phi^{n}(\nu), \mathcal{D})<\varepsilon, \ \forall \nu\in\mathcal{A},
$$
then, using the continuity of $\Phi^{n}$, we have that there exists $\delta>0$ such that
$$
 d_{P}(\mu,\nu)<\delta\Rightarrow d_{P}(\Phi^{n}(\mu),\Phi^{n}(\nu))<\varepsilon.
$$
As $\mathcal{A}$ is dense in $X$, there exists $\nu\in \mathcal{A}$, such that $d_{P}(\nu,\mu)<\delta$.
Finally we get
$$
 n>n_{0}\Rightarrow d_{P}(\Phi^{n}(\mu),\mathcal{D})\leq d_{P}(\Phi^{n}(\nu),
\mathcal{D})+d_{P}(\Phi^{n}(\mu),\Phi^{n}(\nu))<2\varepsilon.
$$
We observe that the last inequality is independent of $\mu\in P(X)$.
\end{proof}

\begin{example}
Consider $X=[0,1]\times[0,1]$ and $T: X\rightarrow X$ given by $T(x,y)=(x,(\frac{1}{2}+\frac{1}{2}x)y)$,
then  $\Lambda=\{(x,y):x=1, \ \mbox{or} \ y=0\}$ is a uniform attractor to $T$.
In fact given
$(x,y)\in X$,
$$
 d(T^{n}(x,y),\Lambda)=d((x,\frac{(1+x)^{n}}{2^n}y),\Lambda)=\min\{1-x,\frac{(1+x)^{n}}{2^n}y\}.
$$

If we take $0<\varepsilon<1$, we have that
  $x\leq \varepsilon$ or $\varepsilon<x$.
If $\varepsilon<x$, then $1-x<\varepsilon$. If $x\leq \varepsilon$, then we can see that

$$
 \frac{(1+x)^{n}}{2^{n}}\leq\frac{(1+\varepsilon)^{n}}{2^{n}}\rightarrow0.
$$
It implies that there exists $n_{0}\in\mathbb{N}$ such that
$$
 n>n_{0}\Rightarrow \frac{(1+x)^{n}}{2^{n}}y\leq\frac{(1+\varepsilon)^{n}}{2^{n}}y<\varepsilon.
$$
Then we conclude that
$$
 n>n_{0}\Rightarrow d(T^{n}(x,y),\Lambda)=\min\{1-x,\frac{(1+x)^{n}}{2^n}y\}<\varepsilon .
$$

On the other hand, if we apply the Theorem \ref{teorema}, we get that the closure of
\begin{align*}
\mathcal{D}:&=\Big\{\sum_{i=1}^{k}\alpha_{i}\delta_{(x_{i},y_{i})}:\sum_{i=1}^{k}\alpha_{i}=1,\\&\
\ (x_{i},y_{i})=(x_{i},0) \ \mbox{or}\  (x_{i},y_{i})=(1,y_{i}), \ \ \alpha_{i}\in [0,1]\cap\mathbb{Q}\mbox { and } k\in\mathbb{N}\Big\},
\end{align*}
is a uniform attractor to $\Phi$.
\end{example}

\begin{example}
(Uniformly hyperbolic attractor) Consider the solid torus $\mathcal{T}=S^{1}\times D^{2}$, where
$S^{1}=[0,1]  \mod 1$ and $D^{2}=\{(x,y)\in\mathbb{R}^{2}:x^{2}+y^2\leq1\}$. We fix $\lambda\in(0\frac{1}{2})$
and define $T:\mathcal{T}\rightarrow\mathcal{T}$ by
$$T(\phi,x,y)=(2\phi,\lambda x +\frac{1}{2}\cos(2\pi\phi),\lambda y+\frac{1}{2}\sin(2\pi\phi) ).$$
The map is injective and stretches by a factor of $2$ in the $S^{1}$-direction, contracts
by a factor of  $\lambda$ in the $D^{2}$-direction,
and wraps the image twice inside $\mathcal{T}$.

The image $F(\mathcal{T})$ is contained in the interior $int(\mathcal{T})$ and
$F^{n+1}(\mathcal{T})\subset int(F^{n}(\mathcal{T}))$. A slice $F(\mathcal{T})\cap \{\phi=c\}$
consists of two disks of radius $\lambda$ centered at diametrically opposite points at distance
$\frac{1}{2}$ from the center of the slice. A slice $F^{n}(\mathcal{T})\cap \{\phi=c\}$
consists of $2^{n}$-disks of radius $\lambda ^{n}$: two disks inside  each of $2^{n-1}$ disks of
$F^{n-1}(\mathcal{T})\cap \{\phi=c\}$.

The set $S=\cap_{n=0}^{\infty} F^{n}(\mathcal{T})$ is called a solenoid. It is a closed
$F$-invariant subset of $\mathcal{T}$ on which $F$
is bijective. The solenoid is a uniform attractor for $F$. Moreover $S$ is a hyperbolic set,
then $S$ is  an example of an  uniformly hyperbolic attractor.

Then, by Theorem \ref{teorema}, the closure of
$$\mathcal{D}:=\Big\{\sum_{i=1}^{k}\alpha_{i}\delta_{q_{i}}:\sum_{i=1}^{k}\alpha_{i}=1,\
\ q_{i}\in S, \  \alpha_{i}\in [0,1]\cap\mathbb{Q}\mbox { and } k\in\mathbb{N}\Big\},$$
is a uniform attractor for $\Phi$.
\end{example}

\section{Topological entropy}\label{entropia topologica}
Here we get a very interesting connection between the topological entropy of the map $T$ and the
topological entropy

We briefly recall the definition of the topological entropy.
\begin{definition}
Let $T:X \rightarrow X$ a continuous map. A subset $A\subset X$ is said $(n,\varepsilon)-$separeted
if any two distinct points $x,y$ satisfy
$$
 d_n(x,y):=\max_{0\leq k\leq n-1} d(T^k(x),T^k(y))\geq \varepsilon.
$$
Each $d_n$ is a metric on $X$, moreover the $d_i$ are all equivalent metrics.
 \end{definition}
Let us denote by $sep(T, n, \vep)$ the maximal cardinality of a $(n, \vep)-$ separated set.
Introducing
$$
 \displaystyle h_{\varepsilon}(T)=\overline{\displaystyle\lim_{n\rightarrow\infty}}\frac{1}{n}\log \ sep(T,n,\varepsilon)
$$
the topological entropy of the map $T$ is then given by
$$
 h(T)=\displaystyle\lim_{\varepsilon\rightarrow 0} h_{\varepsilon}(T).
$$

 Now we can state some results about the topological entropy of the map $\Phi$.

\begin{lemma}
Let $T:X\rightarrow X$ be a continuous map such that $T:X\rightarrow T(X)$ is a homeomorphism .
If  $\displaystyle\lim_{n\rightarrow\infty} T^{n}(x)=p$,
 for all $x\in X$, then $h(T)=0$.
\end{lemma}
\begin{proof}
  We know, by Lemma \ref{entropia}, if $\displaystyle\lim_{n\rightarrow\infty} T^{n}(x)=p$,
  then the sequence $\{T^{n}\}_{n\in\mathbb{N}}$ converges uniformly to the constante map $G\equiv p$.
  Let us take  $A\subset X$  $(N_{\varepsilon},\varepsilon)$-separated with maximum cardinality
  and observe that $A$ is $(n, \varepsilon)$-separated  for all $n\geq N_{\varepsilon}$.
  Moreover if $B$ is $(n, \varepsilon)$-separated, the cardinality of $B$ is at most equals to the cardinality of $A$. Hence
$$
 \displaystyle h_{\varepsilon}(T)=\overline{\displaystyle\lim_{n\rightarrow\infty}}\frac{1}{n}\log \ sep(T,n,\varepsilon)=0,
$$
which implies
$$
 h(T)=\displaystyle\lim_{\varepsilon\rightarrow 0} h_{\varepsilon}(T)=0.
$$
\end{proof}

If we apply the Lemma \ref{ponto fixo para medida} and after apply the Proposition \ref{entropia} we can prove the following:

\begin{theorem}\label{corolario}
If $\displaystyle\lim_{n\rightarrow\infty}T^{n}(x)=p$, then $h(\Phi)=0$, where $h(\Phi) $ is the topological  entropy of $\Phi$.
\end{theorem}

\begin{corollary}
If $T$ is C-Lipschitz with $C<1$, then $h(\Phi)=0$.
\end{corollary}
\begin{proof}
 As $T$ is C-Lipschitz, then $\Phi$ is C-Lipschitz, with $C<1$. Then we have that
$\displaystyle\lim_{n\rightarrow\infty}\Phi^{n}(\mu)=\delta_{p}$, where $p$ is the fixed point for $T$,
 for all $\mu\in\mathcal{P}(X)$. Hence, by Lema \ref{entropia}
$h(\Phi)=0$.
\end{proof}
\begin{remark}
As we proved in Lemma \ref{lipschitz}, if $T$ is $C$-Lipschitz, then $\Phi$ is $C$-Lipschitz.
Hence if $C=1$, then $\Phi$ is non-expansive
and it implies that $h(\Phi)=0$.
\end{remark}
\begin{example}
 Consider the map $T:\mathbb{S}^{1}\rightarrow\mathbb{S}^{1}$ given by $T(x)=x+\alpha$,
 $\alpha $ irrational, then $h(\Phi)=0$. In fact, as $T$ is an isometry we have
 that $\Phi$ is 1-Lipschitz , then $h(\Phi)=0$.
\end{example}

\begin{definition} The set of probability measures supported on a finite set is given by the union
$\mathcal{D}=\cup_{n\geq1}\mathcal{D}_{n}$, where
$$
 \mathcal{D}_{n}=\Big\{\mu=\sum_{i=1}^{n}p_{i}\delta_{x_{i}}:(p_1,...,p_n)\in \mathbb{R}^n,\
    \sum_{i=1}^{n}p_i=1\mbox{ and } x_{i}\in X \Big\},
$$
and  for a fixed $p=(p_1,...,p_n)\in \mathbb{R}^n$ such that $ \sum_{i=1}^{n}p_i=1$ we define the set
$$
 \mathcal{D}_{n}(p)=\Big\{\mu=\sum_{i=1}^{n}p_{i}\delta_{x_{i}}:  x_{i}\in X \Big\}.
$$
\end{definition}

We notice that is possible to make a copy of the space $X$ in the  $\mathcal{P}(X)$ as follows:
$$
 j: X\rightarrow \mathcal{D}_{1}\subset\mathcal{P}(X)
$$
$$
 x\mapsto \delta_{x}.
$$
If we consider $\mathcal{D}_{1}$, we notice that $\Phi(\mathcal{D}_{1})=\mathcal{D}_{1}$, i.e.,
 $\mathcal{D}_{1} $ is $\Phi$-invariant.

\begin{lemma}(See \cite{Infinite  dimensional Analysis} and \cite{Gigli})
$j $ is a homeomorphism onto $\mathcal{D}_{1}$. If we consider the Wasserstein distance, $j $ is an isometry.
\end{lemma}

\begin{lemma}
Let $S:Z\rightarrow Z$ a homeomorphism of a compact metric space.
If $F\subset Z$ is a closed invariant subset of $X$  then
$$
 h(S|_{F})\leq h(S).
$$
\end{lemma}
\begin{proof}
See \cite{Brin}.
\end{proof}

\begin{proposition}
$h(\Phi)\geq h(T)$.
\end{proposition}
\begin{proof}
 We know that $j\circ T(x)=\delta_{T(x)}=\Phi\circ j(x)$, i.e,
$$
 j\circ T=\Phi\circ j.
$$
Hence $T$ is topologically conjugated to $\Phi$ restricted to $\mathcal{D}_{1}$, which implies
$$h(\Phi)\geq h(\Phi|_{\mathcal{D}_{1}})=h(T),$$
because $\mathcal{D}_{1}$ is $\Phi$-invariant.
\end{proof}

We have another important relation between $h(T)$ and $h(\Phi)$. To prove this relation we need some results, which we will not prove.

\begin{lemma}
(Goodwin, 1971) Let X and Y compact Haussdorf spaces and let $T:X\rightarrow X$ and $S:Y\rightarrow Y$ continuous. Then
$$h(T\times S)=h(T)+h(S),$$
where h denotes the topological entropy and $T\times S:X\times Y\rightarrow X\times Y$ is  defined as
$$(T\times S)(x,y)=(T(x),S(y)), \mbox{ for } (x,y)\in X\times Y.$$
\end{lemma}

\begin{theorem}
If $h(T)>0$ then $h(\Phi)=\infty$.
\end{theorem}
\begin{proof} Consider $n\in \N$ and $p\in\R^n$, such that $p=(p_1,..,p_n)$ and $p_i=\displaystyle\frac{2^{i-1}}{2^n-1}$,
and take the set $\mathcal{D}_{n}(p)$.
 We notice that $\mathcal{D}_{n}(p)$ is a closed subset of $\mathcal{P}(X)$, since $\mathcal{D}_{n}(p)=\sum_{i=1}^n p_i\mathcal{D}_{1}$. So we consider a map $\delta_p:X^n\rightarrow \mathcal{D}_{n}(p)$ defined as
$$\delta_p(x_1,...,x_n):=\sum_{i=1}^n p_i\delta_{x_i}.$$
We also consider the map $T^{(n)}:X^n\rightarrow X^n$ defined as
$$T^{(n)}(x_1,..,x_n):=(T(x_1),...,T(x_n))$$
It is not difficult to see that $\delta_p$ and $T^{(n)}$ are continuous,  and they satisfy
$$\Phi\circ\delta_p=\delta_p\circ T^{(n)}.$$
 We claim that $\delta_p$ is injective. In fact if $\delta_p(x)=\delta_p(y)$ and $y\not=x$, then
$$
\Big(\sum_{i=1}^n p_i\delta_{x_i}\Big)(A)=\Big(\sum_{i=1}^n p_i\delta_{y_i}\Big)(A),\  \mbox{for all open}\ A\subset X.
$$
So there is $k$ such that $x_k\not=y_k$. Take an open set $A$
(we can do it because we are assuming $X$ Haussdorf), such that $x_k\in A$
but $y_k\not\in A$.  We consider the set of points  $x_i\in A$, say $\{x_{i_{1}},...,x_{i_{l}}\}\subset\{x_1,...,x_n\}$ that set. Using
the same idea consider the set of points $y_j\in A$, say
$\{y_{j_{1}},...,y_{j_{s}}\}\subset\{y_1,...,y_n\}$ (observe that
$y_k\not\in\{y_{j_{1}},...,y_{j_{s}}\}$). Then we have that
$$
\sum_{t=1}^{l}\displaystyle\frac{2^{i_t-1}}{2^n-1}=
 \Big(\sum_{i=1}^n p_i\delta_{x_i}\Big)(A)=
 \Big(\sum_{i=1}^n p_i\delta_{y_i}\Big)(A)=
  \sum_{m=1}^{s}\displaystyle\frac{2^{j_m-1}}{2^n-1},
$$
and it implies
$$
\sum_{t=1}^{l}2^{i_t-1}=\alpha=\sum_{m=1,j_m\not=k+1}^{s}2^{j_m-1}.
$$
 Then we see that $\alpha \in\N$ has two different representations in base 2, it is a contradiction and we get $\delta_p$ injective. Clearly we have that $\delta_p$ is surjective, then $\delta_p$ is a bijection. As $\delta_p$ is continuous and $X^n$ and $\mathcal{D}_{n}(p)$ are compact (because $X$ is compact and $\mathcal{D}_{n}(p)$ is a closed subset of a compact set) we have that $\delta_p$ is a homeomorphism.  As $\Phi\circ\delta_p=\delta_p\circ T^{(n)}$ and $\delta_p$ is a homeomorphism, $\delta_p$ is a conjugation. Then
$$nh(T)=h(T^{(n)})=h(\Phi|_{\delta_p(X^n)})\leq h(\Phi),$$
as $h(T)>0$ we get the result.
\end{proof}
\begin{corollary}
If $T$ is continuous and $h(T)>0$ then $h(\Phi)=\infty$.
\end{corollary}
\begin{proof} We notice that we did not use the fact that $T$ is a homeomorphism. So we got a homeomorphism $\delta_p:X^n\rightarrow\mathcal{D}_{n}(p)$ such that $\Phi\circ\delta_p=\delta_p\circ T^{(n)}$. It implies that
$h(T^{(n)})\geq h(\Phi|_{\mathcal{D}_{n}(p)})$. By the other hand we have that  $\delta_{p}^{-1}\circ\Phi= T^{(n)}\delta_{p}^{-1}$. It implies that $h(T^{(n)})\leq h(\Phi|_{\mathcal{D}_{n}(p)})$. Finally we get
 $$nh(T)=h(T^{(n)})=h(\Phi|_{\mathcal{D}_{n}(p)})\leq h(\Phi).$$
\end{proof}

\begin{example}
Let $\mathbb{S}^1=\mathbb{R}\slash\mathbb{Z}$ and
consider the map $\phi_d:\mathbb{S}^1\rightarrow\mathbb{S}^1$ defined by
$$\phi_d(x)=dx\ \mbox{mod}\ 1.$$
We know that $h(\phi)=\log d$. Then  if $\Phi$ is the induced map by $\phi_d$, so we have that $h(\Phi)=\infty$.
\end{example}

\section*{Acknowledgments}
 It is a pleasure to thanks Artur Lopes for many discussions concerning the
 topics of this article and also for calling our attention to Kloeckner's paper.

\medskip
Received xxxx 20xx; revised xxxx 20xx.
\medskip
\end{document}